\documentclass[a4paper,12pt]{amsart}
\usepackage{graphicx,amsmath,amsfonts,amssymb,amsthm,mathrsfs}
\usepackage{geometry}
\usepackage{float}
\usepackage{hyperref}
\usepackage{setspace}
\usepackage{quiver}
\usepackage{listings}
\lstdefinelanguage{Magma}{
  keywords={for, while, if, do, then, else, function, return},
  sensitive=true,
  morecomment=[l]{//},
  morecomment=[s]{/*}{*/},
  morestring=[b]",
}
\hypersetup{
    colorlinks=true,
    linkcolor=blue,
    filecolor=magenta,      
    urlcolor=cyan,
    pdftitle={Overleaf Example},
    pdfpagemode=FullScreen,
    }

\newcommand{\rep}[1]{\langle\!\langle {#1}\rangle\!\rangle}

\newcommand{\FF}{{\mathbb F}}

\newcommand{\leg}[2]{\genfrac{(}{)}{}{}{#1}{#2}}

\def\Fq2{{\mathbb F}_{q^2}}
\def\Fp2{{\mathbb F}_{p^2}}

\def\cC{{\mathcal C}}

\def\cH{{\mathcal H}}
\def\cX{{\mathcal X}}
\def\cJ{{\mathcal J}}

\def\Fq2{{\mathbb F}_{q^2}}

\def\cC{{\mathcal C}}

\def\cH{{\mathcal H}}
\def\cX{{\mathcal X}}
\def\cJ{{\mathcal J}}

\numberwithin{equation}{section}
\theoremstyle{plain}
\newtheorem{thm}{Theorem}
\numberwithin{thm}{subsection}
\newtheorem{lem}[thm]{Lemma}
\newtheorem{cor}[thm]{Corollary}
\newtheorem{pro}[thm]{Proposition}

\newtheorem{notation}[thm]{Notation}
\newtheorem{remark}[thm]{Remark}

\newtheorem{question}[thm]{Question}

\newtheorem{example}[thm]{Example}

\def\blfootnote{\xdef\@thefnmark{}\@footnotetext}
\begin{document}
\title{CM theory, maximal hyperelliptic curves, and Chebyshev polynomials}

\author{Saeed Tafazolian and Jaap Top
   }
\date{}
\address{Departamento de Matem\'{a}tica - Instituto de Matem\'{a}tica, Estat\'{i}stica e Computa\c{c}\~ao Cient\'{i}fica
(IMECC) - Universidade Estadual de Campinas (UNICAMP), Rua S\'{e}rgio Buarque de Holanda, 651, Cidade Universit\'{a}ria,  Zeferino Vaz, Campinas, SP 13083-859, Brazil}
\address{Bernoulli Institute for Mathematics, Computer Science, and Artificial Intelligence,\\
Nijenborgh~9\\9747 AG Groningen\\ the Netherlands}
\email{saeed@unicamp.br}
\email{j.top@rug.nl}


   \begin{abstract}
This paper studies hyperelliptic curves $\cH_d$ corresponding to $y^2=\varphi_d(x)$ 
over finite fields, with $\varphi_d(x)$ a Chebyshev polynomial. Starting from the case 
where $d=\ell$ is an odd prime number, new cases $(d,q)$ are presented where $\cH_d$ is maximal
over the finite field $\FF_{q^2}$ of cardinality $q^2$. In addition, new conditions ruling out the possibility
that $\cH_d/\FF_{q^2}$ is maximal for given $(d,q)$, are presented.
The arguments involve a mix of
results on slopes of Frobenius, explicit descriptions
of abelian subvarieties of the jacobian of $\cH_d$ with complex multiplication,
 and a technique from the theory of $2$-descent on
jacobians of hyperelliptic curves. In particular, the
method used here to prove maximality in characteristics $p\equiv 1\bmod 4$ for $d\equiv 1\bmod 4$ a prime number, deserves attention, as it differs from earlier maximality arguments for other curves. Using the new results as
well as extensive calculations with Magma, we pose some
questions. A positive answer would completely classify
the pairs $(q,d)$ resulting in maximality.
   \end{abstract}

\maketitle

{\bf\em Keywords:} finite field, maximal curves, hyperelliptic curves, complex multiplication,
 Chebyshev polynomials, Newton polygon, 2-descent.

 \emph{2000 Mathematics Subject Classification}: 11G20, 11M38, 14G15, 14H25, 14K22.

\section{Introduction}\label{intro}

Given $d\in\mathbb{Z}_{>0}$, the $d$-th
Chebyshev polynomial $\Phi_d(x)\in\mathbb{Z}[x]$ is the unique polynomial such that in
$\mathbb{Z}[x,x^{-1}]$ the equality
\[ \Phi_d\left(x+\frac1x\right)=x^d+\frac1{x^d}\]
holds. For a finite field $\FF_q$
of cardinality $q$ and characteristic
$p$, denote the reduction of
$\Phi_d(x)$ modulo $p$ by $\varphi_d(x)\in\FF_p[x]\subseteq\FF_q[x]$.
Immediate consequences of the definition include
\begin{itemize}
\item[{$*$}] $\varphi_d(x)$ is monic and of degree $d$;
\item[{$*$}] $\varphi_d(0)=\left\{
\begin{array}{rl}
0&\textrm{for}\;d\;\textrm{odd},\\
-2&\textrm{for}\;d\equiv 2\bmod 4,\\
2&\textrm{for}\;d\equiv 0\bmod 4;
\end{array}
\right.$
\item[{$*$}] $\varphi_d(-x)=(-1)^d\varphi_d(x)$;
\item[{$*$}] $\varphi_{de}(x)=\varphi_d(\varphi_e(x))$;
\item[$*$] $\varphi_{d+2}(x)=x\cdot\varphi_{d+1}(x)-\varphi_d(x)$;
\item[$*$] the derivative $\varphi'_d(0)=\left\{
\begin{array}{rl}
0&\textrm{for}\;d\;\textrm{even},\\
d&\textrm{for}\;d\equiv 1\bmod 4,\\
-d&\textrm{for}\;d\equiv 3\bmod 4.
\end{array}
\right.$
\end{itemize}
An additional property follows  (with $\beta\colon t\mapsto t+\frac1t$) from commutativity of the diagram
\[    
\begin{array}{ccc}
\mathbb{P}^1 & \stackrel{t\mapsto t^d}{\longrightarrow}& \mathbb{P}^1 \\
{\Big\downarrow}{\scriptstyle \beta} && {\Big\downarrow}{\scriptstyle\beta} \\
\mathbb{P}^1 & \stackrel{t\mapsto\varphi_d(t)}{\longrightarrow}& \mathbb{P}^1:
\end{array}
    \]
\begin{itemize}
    \item[$*$] $\varphi_d(x)\in\FF_q[x]$ is separable $\Longleftrightarrow\;\;d=1\;\;\text{or}\;\;\gcd(q,2d)=1$.
\end{itemize}
In the remainder of this paper we
consider the complete, regular and absolutely irreducible curve $\cH_d$ over $\FF_q$ corresponding to the affine equation
\[ y^2=\varphi_d(x).\]
Here we assume $q$ is odd and
$\varphi_d(x)$ is separable; in other
words $\gcd(q,2d)=1$. This
implies that $\cH_d$ has genus
$g=g(\cH_d)=\lfloor (d-1)/2\rfloor$.
Our aim is to describe the pairs
$(d,q)$ such that $\cH_d$ is maximal
over $\FF_{q^2}$; by definition this
means that $\#\cH_d(\FF_{q^2})$ attains
the Hasse-Weil upper bound for the number of points on a curve over a
finite field, i.e.,
\[ \#\cH_d(\FF_{q^2})=q^2+1+2g(\cH_d)q.\]

\vspace{\baselineskip}
One of the main results presented here is Thm.~\ref{promax}:
it states that if $\ell\equiv p\equiv 1\bmod 4$, then
$\cH_\ell$ is maximal over $\FF_{p^{\ell-1}}$.
The proof is a mixture of rather different tools, namely
permutation polynomials, some theory of complex multiplication,
and $2$-descent in the arithmetic of jacobians of hyperelliptic
curves.

Maximality results for more general $\cH_d$ are discussed as well. In characteristics $\equiv 3\bmod 4$ a complete
characterization for odd $d$ is provided by Prop.~\ref{3mod4}
(the case $2|d$ is discussed in \cite[Thm.~1.5 and \S5]{TT}).
In characteristics $p\equiv 1\bmod 4$, some necessary
conditions for maximality of $\cH_d$ are presented in \S\ref{subsec1mod4}. Computer experiments resulted
in a simple observation, namely that the only examples
where we have seen maximality of $\cH_d$ in characteristic
$p\equiv 1\bmod 4$, is for $d=\ell^n$ a power of a prime
$\ell\equiv 1\bmod 4$ and moreover $p\bmod d$ a generator
of the cyclic group $(\mathbb{Z}/d\mathbb{Z})^\times$.

The text is organized as follows. Section~\ref{prelim}
discusses necessary background material.
In Section~\ref{primedeg} the curves $\cH_d$ are discussed
for $d=\ell\geq 3$ a prime number.
Next, \S\ref{3mod4} contains complete results for odd $d$
and characteristic $\equiv 3\bmod 4$, and \S\ref{subsec1mod4}
presents an approach towards the characteristic $\equiv 1\bmod 4$ case for general $d$, culminating besides various examples
in Question~\ref{mainquestion}. Finally, the Appendix contains and briefly
describes Magma code used in
this paper.

\section{Preliminaries}\label{prelim}
Maximality of any curve $\cC/\FF_{q^2}$ of genus $g$ is equivalent to the $q^2$-Frobenius endomorphism
$F$ of the jacobian $\cJ(\cC)$ being equal to $[-q]$, the multiplication by $-q$.
In this case $\cJ(\cC)(\FF_{q^2})=\textrm{Ker}(F-\textit{id})=\textrm{Ker}([q+1])\cong 
\left(\mathbb{Z}/(q+1)\mathbb{Z}\right)^{2g}$. So in particular when $q$ is odd and the maximal curve $\cC$ is
a hyperelliptic curve obtained from an equation $y^2=f(x)$ with $f(x)$ separable and of odd degree, then $f(x)$ splits completely in $\FF_{q^2}[x]$: indeed, if $f(\alpha)=0$, then the divisor
$(\alpha,0)-\infty$ yields a point of order $2$ in $\cJ(\cC)$, hence $\alpha\in\FF_{q^2}$.
Maximality of $\cC/\FF_{q^2}$ is also equivalent to the statement that the characteristic polynomial of $F\in\textrm{End}(\cJ(\cC))$ equals $(X+q)^{2g}$. In this case the slopes of the Newton polygon
of this polynomial consist of the singleton set $\{\frac12\}$. Recall that for a polynomial
$a_0X^n+a_{1}X^{n-1}+\ldots+a_{n-1}X+a_n\in\mathbb{Z}[X]$ and $q$ a power of the prime $p$,
these slopes are the slopes of the lower boundary of the convex hull in $\mathbb{R}^2$ of
the points $\left(j, \textrm{ord}_p(a_j)/\textrm{ord}_p(q)\right)$. The vertical scaling factor
$\textrm{ord}_p(q)$ here makes the slopes of Frobenius independent of  extensions of the finite field and corresponding Frobenius. In particular, as will be used below, if $\cC/\FF_p$ is maximal over
some extension $\FF_{p^{2e}}$ then $\frac12$ is the only slope of the $p$-Frobenius on $\cJ(\cC)$.

Regarding the polynomials $\varphi_d(x)$, a well-known property is that under the condition
$\gcd(p^{2n}-1,d)=1$, the map $t\mapsto \varphi_d(t)$ regarded as a map $\varphi_d\colon \FF_{p^n}\to\FF_{p^n}$ is a permutation. At least in the case $d$ is odd, this was
already shown in 1896 by L.E.~Dickson, see \cite[\S54]{Dickson} (note the misprint
in the statement of the result, writing $p^{2n-1}$ instead of $p^{2n}-1$ as is clearly
used in the proof on his next page). Dickson's result in particular implies that
with $m$ the order of $p\bmod d\in\left(\mathbb{Z}/d\mathbb{Z}\right)^\times$ one
has $\#\cH_d(\FF_{p^n})=p^n+1$ for  every $n<m\cdot (3-\gcd(m,2))$. As a consequence, for
these $n$ the coefficient of $X^{2g-n}$ in the characteristic polynomial of the $p$-Frobenius
is $0$.

A final observation that will be important for us, concerns the case that $d=\ell\geq 3$ is
a prime number. As explained in \cite[Proposition~4]{TTV}, the jacobian $\cJ(\cH_\ell)$
contains the CM field $K_\ell:=\mathbb{Q}(i,\zeta_\ell+\zeta_\ell^{-1})$ in its endomorphism algebra;
here $\zeta_\ell$ denotes a primitive $\ell$-th root of unity. The action of $K_\ell$ on
regular translation-invariant $1$-forms on $\cJ(\cH_\ell)$ (the jacobian now considered as an abelian variety
in characteristic $0$) defines a CM type $\Phi_\ell$.
Since $K_\ell$ is a Galois extension of $\mathbb{Q}$, this CM type can be described as
a subset of $\textrm{Gal}(K_\ell/\mathbb{Q})$. Explicitly, $\mathbb{Q}(i,\zeta_\ell)=\mathbb{Q}(\zeta_{4\ell})$ has Galois group
$\textrm{Gal}(\mathbb{Q}(\zeta_{4\ell})/\mathbb{Q})\cong\left(\mathbb{Z}/4\ell\mathbb{Z}\right)^\times\cong \left(\mathbb{Z}/4\mathbb{Z}\right)^\times \times \left(\mathbb{Z}/\ell\mathbb{Z}\right)^\times$.
The subfield $K_\ell$ corresponds to the order two subgroup
$\langle(1,-1)\rangle\subset
\left(\mathbb{Z}/4\mathbb{Z}\right)^\times \times \left(\mathbb{Z}/\ell\mathbb{Z}\right)^\times$.
Hence $\textrm{Gal}(K_\ell/\mathbb{Q})
\cong \left(\mathbb{Z}/4\mathbb{Z}\right)^\times \times\FF_\ell^\times/(\pm 1)$.
As a subset of this, the CM type $\Phi_\ell$ is given as
\[\Phi_\ell=\left\{(1,\pm 1),\,(-1,\pm 2),\,(1,\pm 3),\,\ldots,\,\left((-1)^{(\ell-3)/2},(\ell\pm 1)/2\right)\right\}.\]
Based on work of Shimura and Taniyama
one can describe, as e.g., shown in
\cite[Theorem~3.1]{Blake}, at any prime $p\not\in\{2,\ell\}$ the slopes (of
the Newton polygon of the characteristic polynomial) of Frobenius
on the reduction modulo $p$ of $\cJ(\cH_\ell)$.
Explicitly, the decomposition group
$D_p\subset \textrm{Gal}(K_\ell/\mathbb{Q})$ at the prime $p$
equals the cyclic group
\[D_p = \langle (p\bmod 4, \pm p)\rangle\subset \left(\mathbb{Z}/4\mathbb{Z}\right)^\times \times\FF_\ell^\times/(\pm 1).\]
In terms of this, the slopes of the $p$-Frobenius on $\cJ(\cH_\ell)$ are the fractions
\[ \displaystyle \frac{\#\left(\Phi_\ell\cap D_p\cdot \tau\right)}{\#D_p}\;\;\;\textrm{taken with multiplicity}\;\;\;[\textrm{Gal}(K_\ell/\mathbb{Q}):D_p],\]
with $\tau$ ranging over a set of representatives of the cosets $D_p\backslash \textrm{Gal}(K_\ell/\mathbb{Q})$. A straightforward example illustrating how this
is used, is provided below in Prop.~\ref{3,1}. We note in passing that CM results for
$\cJ(\cH_\ell)$ were also used by Sugiyama, 
see in particular \cite[Example~5.3, Corollary~1.2]{Sugiyama}.
The next session includes, notably in the results of \S\ref{1,1}, additional applications of this kind of argument. It is also used extensively in the calculations
of \S\ref{subsec1mod4}.

\section{Prime degree}\label{primedeg}
Let $\ell$ be an odd prime number. We discuss the curve $\cH_\ell$ over finite fields of
characteristic $p\not\in\{2,\ell\}$. Its genus is $g:=g(\cH_\ell)=(\ell-1)/2$. The tools for this that were introduced in Section~\ref{prelim}
make it natural to distinguish cases according to the possibilities of the pair $\ell\bmod 4,\,p\bmod 4$.
In  the discussion an important role is played by the CM type $\Phi_\ell$. 
\subsection{Case \texorpdfstring{$\ell\equiv 3\bmod 4,\;p\equiv 1\bmod 4$}{ell=3 mod 4, p=1 mod 4}}
\begin{pro}\label{3,1}
 If $\ell\equiv 3\bmod 4$ and $p \equiv 1\bmod 4$, then $\cH_\ell$
is not maximal over any finite field of characteristic $p$.
\end{pro}

\begin{proof}
With notations as introduced
in Section~\ref{prelim}, the condition $p\equiv 1\bmod 4$ implies that in $\textrm{Gal}(K_\ell/\mathbb{Q})\cong 
\left(\mathbb{Z}/4\mathbb{Z}\right)^\times \times\FF_\ell^\times/(\pm 1)$, the decomposition group 
$D_{p}= \langle (1, \pm p)   \rangle$.
 The condition $\ell\equiv 3\bmod 4$
 implies that $\# D_p$ is odd. Hence for any $\tau\in \textrm{Gal}(K_\ell/\mathbb{Q})$, the corresponding slope 
$\frac{\#(\Phi_\ell\,\cap\,D_p\cdot \tau)} {\# D_p}$ differs from $1/2$.
This violates maximality of $\cH_\ell$ in characteristic $p$.
\end{proof}
\subsection{Case \texorpdfstring{$\ell\equiv 3\bmod 4,\; p\equiv 3\bmod 4$}{ell=3 mod 4, p=3 mod 4}}\label{3,3}
In this situation we have the following result.
\begin{pro}\label{pro3,3}
 Let $\ell\equiv 3\bmod 4$ be a prime number. For every prime $p\neq \ell$ with $p \equiv 3\bmod 4$, the powers $q=p^n$ such that $\cH_\ell$
is maximal over $\FF_{q^2}$ are given by
$n=km$ with $m$ the order of $\pm p\bmod\ell$ in the group $\FF_\ell^\times/(\pm 1)$ and $k$ odd.
\end{pro}
\begin{proof}
Take $\ell, p,m$ as in the formulation
of the theorem.
Since the group $\FF_\ell^\times/(\pm 1)$ has odd order, also $m$ is odd.
By definition $q:=p^m\equiv \pm 1\bmod \ell$, and since $m$ is odd, $q\equiv -1\bmod 4$.
As shown in \cite[Theorem~1.6]{TT}, this implies that $\cH_\ell$ is maximal over $\FF_{q^2}$.

We claim that the finite fields $\FF_{p^{2n}}$ over which $\cH_\ell$ is maximal, are
precisely the ones with $n$ an odd multiple of $m$. Indeed, maximality over $\FF_{p^{2m}}$
means that the $p^{2m}$-Frobenius $F$ on $\cJ(\cH_\ell)$ equals $[-p^m]$.
So if $n=km$ with odd $k$, then the $p^{2n}$-Frobenius $F^k$ equals $[-p^m]^k=[-p^n]$,
implying that $\cH_\ell$ is maximal over $\FF_{p^{2n}}$.
Vice versa, maximality over $\FF_{p^{2n}}$ for some $n$ shows, using the same argument
and the fact that $m$ is odd, that $\cH_\ell$ is maximal over $\FF_{p^{2nm}}$.
The Frobenius in this case is $F^n=[-p^m]^n$, and since it has to be equal to $[-p^{mn}]$
we conclude that $n$ is odd. Now \cite[Theorem~3.6, (\textit{ii})$\Rightarrow$(\textit{iii})]{GSJ}
applied to $q=p^n$ shows $p^n\equiv \pm 1\bmod \ell$ and therefore $n$ is a multiple of
the order $m$ of $\pm p\bmod \ell$
in the group $\FF_\ell^\times/(\pm 1)$. 
This completes the proof.
\end{proof}

Reformulating Proposition~\ref{pro3,3}, in
the given situation the powers $q$ of $p$ such that $\cH_\ell$ is maximal over $\FF_{q^2}$,
are precisely those satisfying the two conditions $q\equiv 3\bmod 4$ and $q\equiv \pm 1 \bmod \ell$.

\subsection{Case \texorpdfstring{$\ell\equiv 1\bmod 4,\;p\equiv 3\bmod 4$}{ell=1 mod 4, p=3 mod 4}}\label{1,3}
\begin{pro}\label{pro1,3}
Suppose $\ell\equiv 1\bmod 4$ is a prime number.
A necessary condition for maximality of
$\cH_\ell$ over some finite field of characteristic $p\equiv 3\bmod 4$, is that
the order $m$ of $\pm p\bmod \ell\in\FF_\ell^\times/(\pm 1)$ is odd.

In case $m$ is odd and $q=p^n$, the curve
$\cH_\ell$ is maximal over $\FF_{q^2}$ precisely when
$n$ is an odd multiple of $m$.
\end{pro}
\begin{proof} Consider primes $\ell, p$ satisfying $\ell\equiv 1\bmod 4$ and $p\equiv 3\bmod 4$.
The arguments presented below for this case are partially similar
to those already used in \S\ref{3,3}. 

Suppose $q=p^n$ is such that $\cH_\ell/\FF_{q^2}$ is maximal. As explained in Section~\ref{prelim}, this implies in
particular that $\varphi_\ell(x)\in\FF_{q^2}[x]$
splits completely. With $\zeta_{4\ell}\in\overline{\FF}_p$ a
primitive $4\ell$-th root of unity,
this means that $\FF_p(\zeta_{4\ell}+\zeta_{4\ell}^{-1})\subseteq \FF_{q^2}$.
Since clearly $\FF_{p^2}\subseteq\FF_{q^2}$,
then \cite[Lemma~3.5]{GSJ} yields
\[
\zeta_\ell\in \FF_{p^2}(\zeta_{4\ell}+\zeta_{4\ell}^{-1})\subseteq\FF_{q^2}
\]
and therefore $\ell|(q^2-1)$.
Using \cite[Theorem~3.2]{GSJ}, the latter
divisibility gives rise to an over
$\FF_{q^2}$ defined isogeny
\[
\cJ(\cX)\sim \cJ(\cH_\ell)\times \cJ(\cH_\ell)\times E
\]
where $\cX$ denotes the smooth projective curve
corresponding to $y^2=x^{2\ell+1}+x$ and
$E$ is the elliptic curve given by $y^2=x^3+x$.
Maximality of $\cH_\ell/\FF_{q^2}$ together
with the observation that since $p\equiv 3\bmod 4$,
we have that $E$ is supersingular,
now shows that $\cX$ is supersingular, i.e.,
any zero of the characteristic polynomial
of the $p$-Frobenius on $\cJ(\cX)$ is of the
form $\zeta\sqrt{p}$ with $\zeta$ a root
of unity. By \cite[Theorem~2]{KW}, one
concludes that the order $m$ of $p\in\FF_\ell^\times/(\pm 1)$ is odd.
This means $\#D_p\equiv 2\bmod 4$,
and $p^m\equiv -1$ or $1+2\ell\bmod 4\ell$.
Now \cite[Theorem~3.6, (\textit{iii})$\Rightarrow$(\textit{ii})]{GSJ} shows
that $\cH_\ell/\FF_{p^{2m}}$ is maximal.
Maximality over odd degree extensions of this
field then holds as well.

As $m$ is odd, maximality of $\cH_\ell/\FF_{q^2}$ implies that $\cH_\ell$ is maximal
over $\FF_{p^{2nm}}$. Viewing the latter field
as an extension of $\FF_{p^{2m}}$ then reveals
as in \S\ref{3,3} that $n$ is odd.
We have $p^{2n}=q^2\equiv 1\bmod \ell$,
hence $p^n\equiv \pm 1\bmod \ell$ which shows
that $n$ is an odd multiple of the order
$m$ of $p\in\FF_\ell^\times/(\pm 1)$.
\end{proof}
\begin{remark}{\rm
    A generalization of Prop.~\ref{pro1,3} to $\cH_d$ for general
    odd $d$ is presented below in Prop.~\ref{3mod4}. 
}\end{remark}

\begin{remark}
{\rm The maximal curves $\cH_\ell$ described in
Proposition~\ref{pro1,3} are covered by the Hermitian curve. 
This follows as a special case of composing maps presented in the proofs of
\cite[Thm.~1]{Taf} and \cite[Thm.~3.2]{GSJ}.}
\end{remark}

\subsection{Case \texorpdfstring{$\ell\equiv 1\bmod 4,\;p\equiv 1\bmod 4$}{ell=1 mod 4, p=1 mod 4}}\label{1,1}
Given primes $\ell\neq p$ with $\ell\equiv p\equiv 1\bmod 4$, consider the decomposition group
$D_{p}= \langle (1, \pm p)   \rangle\subset \left(\mathbb{Z}/4\mathbb{Z}\right)^\times \times\FF_\ell^\times/(\pm 1)$.
As remarked in Section~\ref{prelim}, a necessary condition for existence of a finite extension of $\FF_p$
over which $\cH_\ell$ is maximal, is that $\#D_p$ is even.
In the case under consideration this is equivalent to the order of $p\bmod\ell\in\FF_\ell^\times$ being a
multiple of $4$. We now consider the two extreme cases of this.

\vspace{\baselineskip} \noindent
{\bf Maximal order.} The result here is as follows; note the novel technique to show maximality in the given situation.
\begin{thm}\label{promax}
 Let $\ell\equiv p\equiv 1\bmod 4$ be distinct prime numbers and assume $p$ is a primitive
 root modulo $\ell$.

 Then for $q=p^a$ the curve $\cH_\ell$
 is maximal over $\FF_{q^2}$ if and only if
 $a=b\cdot (\ell-1)/2$ with $b$ a positive odd integer.
\end{thm}
\begin{proof}
The conditions on $\ell,p$ mean that $D_p=\{1\}\times \FF_\ell^\times/(\pm 1)$.
Then $\#(\Phi_\ell\cap D_p)=(\ell-1)/4$ (count the odd integers in $1,2,\ldots,(\ell-1)/2$). 
Similarly 
$\#\left(\Phi_\ell\cap D_p\cdot(-1,\pm 1)\right)=(\ell-1)/4$. Hence the Newton polygon of the $p$-Frobenius
on $\cJ(\cH_\ell)$ has $\frac12$ as its only slope. We will determine the corresponding characteristic
polynomial. 

The assumption on the order of $p\bmod \ell$ implies that $\varphi_\ell(x)$ is a permutation polynomial on $\FF_{p^a}$ and hence $\#\cH_\ell(\FF_{p^a})=p^a+1$ for
every $a$ satisfying
$1\leq a<(\ell-1)/2=g$. So the characteristic polynomial of the $p$-Frobenius has the form
\[
X^{2g}+nX^g+p^g
\]
with $g=g(\cH_\ell)=(\ell-1)/2$ and $n\in\mathbb{Z}$. Moreover, since its Newton polygon has $\frac12$ as
its only slope, $p^{(\ell-1)/4}|n$.
The characteristic polynomial therefore is
\[
X^{2g}+mp^{g/2}X^g+p^g
\]
for some integer $m$. All zeros in $\mathbb{C}$
of this polynomial have absolute value $\sqrt{p}$, and hence $m^2-4\leq 0$,
so $m\in\{-2,-1,0,1,2\}$. Evaluating at $X=1$
and using $p\equiv 1\bmod 4$ yields 
\[\#\cJ(\cH_\ell)(\FF_p)=1+mp^{g/2}+p^g\equiv (2+m)\bmod 4.
\]
We claim that $\#\cJ(\cH_\ell)(\FF_p)\equiv 2\bmod 4$ (which shows $4|m$ and hence $m=0$).
Indeed, the divisor $(0,0)-\infty$ on $\cH_\ell$ yields
a rational point of order $2$ in $\cJ(\cH_\ell)$.
Any other rational point of order $2$
would correspond to a factorization of
$\varphi_\ell(x)\in\FF_p[x]$ into two
factors, with neither factor a constant multiple
of $x$. Write $\phi_\ell(x)=x\cdot \psi_\ell(x)$. Any zero of $\psi_\ell(x)$ is
of the form $\zeta_{4\ell}+\zeta_{4\ell}^{-1}$
with $\zeta_{4\ell}$ a primitive $4\ell$-th
root of unity. We will show that
$[\FF_p(\zeta_{4\ell}+\zeta_{4\ell}^{-1})\,:\,\FF_p]=\ell-1$. Indeed, using the same notation
for primitive roots of unity in characteristic $0$, 
the subfield $\mathbb{Q}(\zeta_{4\ell}+\zeta_{4\ell}^{-1})\subset \mathbb{Q}(\zeta_{4\ell})$ has Galois group
$(\mathbb{Z}/4\ell\mathbb{Z})^\times/(\pm 1)\cong
\left((\mathbb{Z}/4\mathbb{Z})^\times 
\times \FF_\ell^\times\right)/\langle (-1,-1)\rangle$.
Our assumption that $p\equiv 1\bmod 4$ is a primitive root modulo $\ell$ means that
the decomposition group $\tilde{D}_p=\langle(1,p)\rangle$ at $p$ in this
Galois group has size $\ell-1$.
As $\tilde{D}_p\cong\textrm{Gal}(\FF_p(\zeta_{4\ell}+\zeta_{4\ell}^{-1})\,/\,\FF_p)$,
this proves the assertion regarding the degree
of $\FF_p(\zeta_{4\ell}+\zeta_{4\ell}^{-1})$.
So $\psi_\ell(x)\in\FF_p[x]$ is irreducible,
which means that $(0,0)-\infty$ gives rise to
the only point of order $2$ in $\cJ(\cH_\ell)(\FF_p)$.
Hence to show $\#\cJ(\cH_\ell)(\FF_p)\equiv 2\bmod 4$, it remains to prove that
$\cJ(\cH_\ell)(\FF_p)$ contains no point of
order $4$, i.e., no point $P$ such that
$2P$ corresponds to the divisor class of
$(0,0)-\infty$. To verify that indeed
no such $P$ exists, a well-known
technique from $2$-descent computations can be used:
there is a homomorphism
\[
\delta\colon \cJ(\cH_\ell)(\FF_p)\,/\, 2\cJ(\cH_\ell)(\FF_p)
\longrightarrow \FF_p^\times/{\FF_p^\times}^2
\]
given, in terms of divisor classes, by
\[
\delta\colon [(\alpha,\beta)-\infty]
\mapsto \left\{
\begin{array}{rl}
    \alpha{\FF_p^\times}^2 & \textrm{if}\;\;\alpha\neq 0,\\
    \ell\,{\FF_p^\times}^2 & \textrm{if}\;\;\alpha=0;
\end{array}
\right.
\]
see \cite[Lemma~2.2]{Schaefer} or
\cite[Lemma~4.3.(2)(ii)]{Stoll}.
Here we use that $p\equiv 1\bmod 4$
and therefore the properties of
Chebyshev polynomials listed in
Section~\ref{intro} yield $\psi_\ell(0)=\varphi'_\ell(0)=\ell$.
Since $p\equiv 1\bmod 4$ is
a primitive root modulo $\ell$, the Legendre symbol satisfies
$\leg{\ell}{p}=\leg{p}{\ell}\equiv
p^{(\ell-1)/2}\bmod \ell=-1\bmod \ell$.
This shows that $\delta([(0,0)-\infty])$ is nontrivial, and one concludes that
$[(0,0)-\infty]\not\in 2\cJ(\cH_\ell)(\FF_p)$.

As a consequence, the $p$-Frobenius
on $\cJ(\cH_\ell)$ has characteristic
polynomial
\[ X^{2g}+p^g.\]
Hence for $q=p^a$ (and primes $p\neq \ell$ such that $p\equiv \ell\equiv 1\bmod 4$ with $p$ a primitive root modulo $\ell$) the curve
$\cH_\ell$ is maximal over $\FF_{q^2}$ precisely when $a=bg=b(\ell-1)/2$ with $b$ a positive odd integer, which is what we wanted to show.
\end{proof}

\begin{remark}{\rm As shown in
\cite[Prop.~2.3]{CH}, the shape of the
characteristic polynomial in the above proof (namely $X^{2g}+nX^g+p^g$) 
implies that $\cJ(\cH_\ell)$
is $\FF_{p^g}$-isogenous to
the $g$-fold product of some
elliptic curve. Indeed, this is
a simple consequence of classical work
of Tate \cite[Thm.~1(c)]{T}.

One of the ingredients in our
proof is to use the fact that
$\varphi_\ell(x)$ is a permutation polynomial over various $\FF_{p^a}$, in order
to restrict the possibilities
for the characteristic polynomial of the $p$-Frobenius. This idea also appears
in work of \"{O}zbudak \cite[Section~2]{OZ}.
}\end{remark}
\vspace{\baselineskip} \noindent
{\bf Minimal order.} Here we discuss the case
that the prime number $p\equiv 1\bmod 4$
has order $4$ in $\FF_\ell^\times$.
For convenience when studying the CM-type $\Phi_\ell$, the following notation
is introduced. 

\begin{notation}
 Given $a\in\FF_\ell^\times$, denote by $\rep{a}\in\mathbb{Z}_{>0}$ the smallest positive integer $n$
 with the property that $n\bmod \ell \in\{a,\,-a\}$.
\end{notation}
Note that for $a\in\FF_\ell^\times$ one has $1\leq \rep{a}=\rep{-a}\leq (\ell-1)/2$.
The CM type $\Phi_\ell$ can be described, as a subset of 
$\left(\mathbb{Z}/4\mathbb{Z}\right)^\times \times\FF_\ell^\times/(\pm 1)$,
by
\[
\begin{array}{rcl}
(1,\pm a)\in\Phi_\ell & \Longleftrightarrow & \rep{a}\;\textrm{is odd},\\
(-1,\pm a)\in\Phi_\ell & \Longleftrightarrow & \rep{a}\;\textrm{is even}.
\end{array}
\]
\begin{pro}\label{promin}
  Given distinct prime numbers $\ell\equiv p\equiv 1\bmod 4$ such that $\ell\neq 5$ and $p\bmod\ell\in\FF_\ell^\times$ has order $4$,
  the curve $\cH_\ell$ is not maximal over any
  finite extension of $\FF_p$.
\end{pro}
\begin{proof} Write $i:=p\bmod \ell$. The decomposition group
$D_p$ equals $\{(1,\pm 1),\,(1,\pm i)\}$ and we are
interested in the possible slopes
\[
\frac{\#\left(\Phi_\ell\cap D_p\tau\right)}{\#D_p}\in\{0,\,1/2,\,1\}
\]
for $\tau\in \left(\mathbb{Z}/4\mathbb{Z}\right)^\times \times\FF_\ell^\times/(\pm 1)$.
We claim that for any $p,\ell$
as described here, with $\ell> 5$, a slope $\neq 1/2$ occurs.

Indeed, write $a:=\rep{i}<\ell/2$, the smallest integer $>0$
such that $a^2\equiv -1\bmod \ell$.
If $a$ is odd, then $D_p\subset \Phi_\ell$
provides a slope $1$. Hence we may and will
assume $a$ is even. Then 
\[D_p\cdot(1,\pm 2)=\{(1,\pm 2),\,(1,\pm 2a)\}
\]
where $2<\ell/2$ and $2<2a<\ell$.
Hence if $2a<\ell/2$ we obtain a slope $0$.
This reduces our claim to the case of an
even $a$ with $\ell/4<a<\ell/2$. Here, take
\[
D_p\cdot(1,\pm 4)=\{(1,\pm 4),\,(1,\pm4a)\}.
\]
Since $\ell>5$ and $\ell\equiv 1\bmod 4$,
we have $\rep{4}=4$. The smallest representants
of the classes $\pm 4a\bmod \ell$ are $4a-\ell$
and $2\ell-4a$. This again yields a slope $0$,
unless $4a-\ell<2\ell-4a$.
So what remains is the case $a$ even with
$\frac{\ell}{4}<a<\frac{3}{8}\ell$.
Under these conditions we consider
\[
D_p\cdot(1,\pm 5)=\{(1,\pm 5),\,(1,\pm 5a)\}.
\]
As before, $\rep{5}=5$. The smallest positive
representants of $\pm 5a\bmod\ell$ are $5a-\ell$
and $2\ell-5a$. This results in a slope $1$
unless $\frac{3}{10}\ell<a<\frac{3}{8}\ell$.
Assuming these inequalities, take
\[
D_p\cdot(1,\pm 3)=\{(1,\pm 3),\,(1,\pm 3a)\}.
\]
In case $3a<\ell$, the smallest positive
representant of $\pm 3a\bmod \ell$ is $\ell-3a$
and this leads to a slope $1$.
In the remaining case $\ell<3a<\frac{9}{8}\ell$,
clearly $3a-\ell<2\ell-3a$ so also here a slope
$1$ is obtained. This proves the claim.
The proposition is an immediate consequence.
\end{proof}

\begin{cor}\label{corol}Let $\ell$ be a prime number of the
form $\ell=4k+1$ with $k=1$ or $k$ prime.
The only characteristics $p\neq \ell$ with $p\equiv 1\bmod 4$ such that $\cH_\ell$ is maximal over some finite
extension of $\FF_p$ are the ones with the property
that $p$ is a primitive root modulo $\ell$.
\end{cor}

\begin{example}{\rm
The smallest primes $\ell\equiv 1\bmod 4$ are
$5$ and $13$. These cases are covered
as instances of Propositions~\ref{promax},~\ref{promin} and 
Corollary~\ref{corol}.

 For $\ell=17$ and $17\neq p\equiv 1\bmod 4$,
 the only case of a decomposition group $D_p$
 of even order not discussed in the
 Propositions~\ref{promax}, \ref{promin} is
 \[
 D_p=\{(1,\pm1),\,(1,\pm 2),\,(1,\pm 4),\,(1,\pm 8)\}.
 \]
 This occurs whenever $p\equiv 25,\,45,\,49,\,53\bmod 68$.
 Note that $D_p\cap\Phi_\ell=\{(1,\pm 1)\}$,
 so one of the resulting slopes is $1/4$.
 So $\cH_{17}$ is not maximal over any finite
 extension of $\FF_p$ for such $p$.
 In fact, $1/4$ and $3/4$ are the only slopes in this case, each with length $8$. So the
 $p$-rank of $\cJ(\cH_{17})$ is $0$ in the
 described situation, whereas this Jacobian
 is not supersingular.

 A next case not covered completely by the results obtained
 so far, is $\ell=37$. Here apart from a maximal and
 minimal even case, also $\#D_p=6$ occurs,
 explicitly
 \[
 D_p=\{(1,\pm 1),\,(1,\pm6),\,(1,\pm 8),\,(1,\pm10),\,(1,\pm 11), (1,\pm 14)\}.
 \]
 This is the subgroup of $\{1\}\times \FF_\ell^\times/(\pm 1)$
 generated by $(1,\pm 8)$. Since $D_p\cap \Phi_\ell$ leads
 to a slope $1/3$, also here a necessary (and sufficient)
 condition for maximality of $\cH_\ell$ in
 characteristic $p\equiv 1\bmod 4$ is that $p$ is a primitive root modulo $\ell$.

 As a last example, consider $\ell=41$. Here the additional
 cases
 \[D_p=\{(1,\pm 1),\,(1,\pm3),\,(1,\pm 9),\,(1,\pm 14)\}\]
 and
 \[D_p=\left\{(1,\pm a)\;:\;a\in\{1,2,4,5,8,9,10,16,18,20\}\right\},\]
 coming from an element of order $4$ resp. $10$ in
 $\FF_{41}^\times/(\pm 1)$, turn up.
 The first one leads to a Newton polygon with slopes
 $1/4,\,3/4$ and the second one results in slopes
 $3/20,\,17/20$. Hence again $p$ a primitive root modulo
 $\ell$ is a necessary and sufficient condition
 for maximality of $\cH_\ell$ over some field
 of characteristic $p\equiv 1\bmod 4$.

 We do not know whether this is true in general. Using
 the {\tt slopes} Magma code
 presented on p.~\pageref{PrimeMagma} we
 verified that it holds for
 every prime $\ell\equiv 1\bmod 4$, $\ell\leq 101$.
}\end{example}
\begin{remark}{\rm
Similar to the proof of Prop.~\ref{promin}, we verified
for primes $\ell$, $p\equiv 1\bmod 4$ such
that $p\bmod \ell\in\FF_\ell^\times$ has order $8$
(i.e., the decomposition group $D_p\subset\text{Gal}(\mathbb{Q}(i,\zeta_\ell+\zeta_\ell^{-1})/\mathbb{Q})$
has order $4$), a slope $\neq 1/2$ at $p$
on $\cJ(\cH_\ell)$ occurs. Hence also in these cases
maximality of $\cH_\ell$ in characteristic $p$ does not happen.
The results motivate the following question.}\end{remark}
\begin{question}\label{prime1mod4case}
  Given distinct primes $\ell\equiv p\equiv 1\bmod 4$,
  is it true that maximality of $\cH_\ell$ in characteristic $p$ is {\em only} possible in the situation of Thm.~\ref{promax} (i.e., when $p$ is a primitive root modulo $\ell$)?
\end{question}

\begin{remark}
{\rm When $p \equiv 1 \pmod 4$ is a primitive root modulo the prime $\ell\equiv 1\bmod 4$ and $q=p^{(\ell-1)/2}$, we
do not know whether the maximal curve $\cH_\ell/\FF_{q^2}$ is covered by the Hermitian curve.
As a small, explicit example: consider $\ell=5$ and $q=13^2=169$. The genus $2$ curve
$\cH_5\colon y^2=x^5-5x^3+x$ is maximal over $\FF_{13^4}$. Is it over this field covered by the Hermitian curve?

Note that Giulietti and Korchm\'aros \cite{GK} constructed the first example of a maximal curve over 
a finite field which is not covered by the Hermitian curve.}
\end{remark}
\section{General odd degree (some results and calculations)}

Throughout this section, $d\in\mathbb{Z}_{\geq 3}$ will be odd.
We will first discuss necessary
and sufficient conditions for
maximality of $\cH_d$ over
finite fields of characteristic
$p\equiv 3\bmod 4$. The
case $p\equiv 1\bmod 4$ is then
discusses in a separate subsection.

\subsection{Characteristic \texorpdfstring{$3$}{3} modulo \texorpdfstring{$4$}{4}}
 Combining results from our papers \cite{TT} and \cite{GSJ}
 with work by Kodama and Washio \cite{KW} leads to the following.
 \begin{pro}\label{3mod4}
     Let $d\in\mathbb{Z}_{\geq 1}$ be odd and let $p\equiv 3\bmod 4$ be a prime number.

     Then $\cH_d$ is maximal over some finite field of
     characteristic $p$ if and only if the subgroup
     $\langle \bar{p}\rangle
     \subseteq (\mathbb{Z}/4d\mathbb{Z})^\times$ contains either $\overline{-1}$ or $\overline{1+2d}$.

     If these equivalent conditions are satisfied,
     then writing $q=p^n$ and
     $2k=\text{ord}(\bar{p})$, 
     it holds that $k$ is an odd integer and one has
\[
\cH_d/\Fq2\;\text{is maximal}\;
\Longleftrightarrow \; n=km\;
\text{with}\;m\in\mathbb{Z}_{\geq 1}\;\text{odd}\;
\Longleftrightarrow \; q\equiv -1,\,1+2d\bmod 4d.
\]
 \end{pro}
 \begin{proof}
 We start by showing the two implications that appear in the first assertion.\\
 $\Leftarrow$: Take
 $q:=p^n$
 such that $\bar{q}\in\{\overline{-1},\overline{1+2d}\}$.
 Then by \cite[Thm.~1.6]{TT}, $\cH_d/\Fq2$ is maximal.\\
 $\Rightarrow$: Let $\cX$ be the curve
 defined by $y^2=x^{2d+1}+x$
 and let $E$ be the elliptic curve with equation $y^2=x^3+x$. As is shown in
 \cite[Rem.~5.5]{TT}, over
 $\overline{\FF}_p$ an
 isogeny $\cJ(\cX)\sim\cJ(\cH_d)^2\times E$ exists. The assumptions that
 $\cH_d$ is maximal and that $p\equiv 3\bmod 4$ therefore imply
 that $\cJ(\cX)$ is supersingular. Then
 \cite[p.~199, Cor. to Thm.~2]{KW} implies that maximality
 occurs for $\cX$ over some finite field of characteristic $p$, and moreover either $\overline{-1}$ or $\overline{1+2d}$ is in $\langle \bar{p}\rangle$.

 To show the remaining parts, first note that
 $(\mathbb{Z}/4d\mathbb{Z})^\times\cong (\mathbb{Z}/4\mathbb{Z})^\times \times (\mathbb{Z}/d\mathbb{Z})^\times$; under the
 standard isomorphism $\bar{p}$
 corresponds to the pair
 $(p\bmod 4,\,p\bmod d)$ in
 the latter group.
 Since $p\equiv 3\bmod 4$, this means that $\text{ord}(\bar{p})$ is even.
 Writing this order as $2k$,
 it follows that $\bar{p}^k$ is
 the unique element of order $2$
 in $\langle \bar{p}\rangle$,
 which by assumption equals one of $\overline{-1},\,\overline{1+2d}$.
 In particular $p^k\equiv 3\bmod 4$, so $k$ is odd.
 Furthermore, $q=p^n\equiv p^k\bmod 4d$ precisely when
 $n$ is an odd multiple of $k$.
 The proof of ``$\Leftarrow$''
 shows that $\cH_d/\Fq2$ is
 maximal. If $\cH_d$ is maximal
 over some $\FF_{p^{2f}}$, then $\varphi_d(x)$ splits completely in $\FF_{p^{2f}}[x]$,
 hence $\Fp2(\zeta_{4d}+\zeta_{4d}^{-1})\subseteq \FF_{p^{2f}}$. Using
 \cite[Lemma~3.5]{GSJ} it follows that $p^{2f}\equiv 1\bmod d$.
 Since also $p^{2f}\equiv 1\bmod 4$, one concludes that $2k|2f$, hence $k|f$.
 Maximality of $\cH_d$ over both $\FF_{p^{2k}}$ and its extension $\FF_{p^{2f}}$
 implies that $f$ is an odd multiple of $k$. This completes the proof.
 \end{proof}
\begin{remark}{\rm
    A crucial ingredient in the proof presented here is the result by Kodama and Washio \cite{KW}
    concerning $\cX\colon y^2=x^{2d+1}+x$. They showed that the three assertions
    \begin{itemize}
        \item[(a)] $\cJ(\cX)$ is supersingular in characteristic $p$;
        \item[(b)] $\cX$ is minimal over some finite field of characteristic $p$;
        \item[(c)] $\cX$ is maximal over some finite field of characteristic $p$,
    \end{itemize}
    are equivalent. Their proof involves calculations with Jacobi sums. Similar ideas
    may be found in the short classical note \cite{TS} by Tate and Shafarevich, who attribute
    it to A.~Weil.
}\end{remark}
If $\ell|d$ for some prime $\ell\equiv 3\bmod 4$, then 
the equality $\varphi_d(x)=\varphi_\ell(\varphi_{d/\ell}(x))$ leads to
a nonconstant morphism
$\cH_d\to\cH_\ell$. As a
consequence, using Proposition~\ref{3,1},
maximality of $\cH_d$ cannot
occur over finite fields of
characteristic $p\equiv 1\bmod 4$. Hence for such odd $d$,
Proposition~\ref{3mod4} describes all cases of maximality in every
characteristic. 
\begin{example}{\rm
  Consider $d=15$. The discussion above combined with Proposition~\ref{3mod4} shows that
  maximality of $\cH_{15}$ occurs in characteristic $p$ precisely when $p\equiv 3\bmod 4$ and
  $\langle \bar{p}\rangle\subset (\mathbb{Z}/60\mathbb{Z})^\times$ contains one of
  $\overline{-1},\overline{31}$.
  This translates into 
  \[ p\equiv -1,\,31 \bmod 60.\]
  In these cases $\cH_{15}/\FF_{p^n}$ is maximal if and only if $n\equiv 2\bmod 4$.
  Note in particular that, for example, when $p\equiv 11\bmod 60$ then both
  $\cH_3$ and $\cH_5$ are maximal over $\FF_{p^2}$, although maximality of
  $\cH_{15}$ cannot occur in such characteristic. Some CM theory discussed in
  the next section turns out to explain this.
}\end{example}

\subsection{Characteristic \texorpdfstring{$1$}{1} modulo \texorpdfstring{$4$}{4}}\label{subsec1mod4}
Combining \cite[Thm.~1.5]{TT} and Prop.~\ref{3,1} results in the observation that a
necessary condition for maximality of $\cH_d$ in characteristic $p\equiv 1\bmod 4$,
provided $d>2$, is that all prime factors of $d$ are $1\bmod 4$.
As the case where $d$ itself is prime, is already discussed in Section~\ref{primedeg},
this means that new such cases have either $\ell_1\ell_2|d$ or $\ell^n|d$ with
$\ell_1\neq \ell_2, \ell$ primes congruent to $1\bmod 4$ and $n\geq 2$.
The elementary properties of the polynomials $\varphi_d(x)$ then yield
either $\cH_d\to\cH_{\ell_1\ell_2}$ or $\cH_d\to\cH_{\ell^n}$ a non-constant morphism.
Therefore, a necessary condition for maximality of $\cH_d$ in characteristic $p$
is (depending on the given $d$) maximality of $\cH_{\ell_1\ell_2}$ or of $\cH_{\ell^n}$.
We focus on these two types of curves.
\begin{lem}\label{ell1ell2}
    Let $\ell_1\neq \ell_2$ be odd prime numbers. The maps $\cH_{\ell_1\ell_2}\to\cH_{\ell_1}$
    and $\cH_{\ell_1\ell_2}\to\cH_{\ell_2}$ coming from $\varphi_{\ell_1\ell_2}(x)=\varphi_{\ell_1}(\varphi_{\ell_2}(x))=
    \varphi_{\ell_2}(\varphi_{\ell_1}(x))$ lead to a $\mathbb{Q}$-isogeny
    \[ \cJ(\cH_{\ell_1\ell_2})\sim \cJ(\cH_{\ell_1})\times \cJ(\cH_{\ell_2})\times A\]
    for some abelian variety $A/\mathbb{Q}$ of dimension $(\ell_1-1)(\ell_2-1)/2$.

    Write $d=\ell_1\ell_2$. In characteristic $0$, this $A$ has complex multiplication with endomorphism algebra
    $\mathbb{Q}(i,\zeta_d+\zeta_d^{-1})$. Identifying $\mbox{\rm{Gal}}(\mathbb{Q}(i,\zeta_d+\zeta_d^{-1})/\mathbb{Q})$
    with $(\mathbb{Z}/4\mathbb{Z})^\times\times (\mathbb{Z}/d\mathbb{Z})^\times/(\pm 1)$,
    the corresponding CM type $\Phi$ is described by the following subset of this Galois group:
    \[
    \Phi=\left\{\left((-1)^{m+1},\pm m\right)\;:\; 1\leq m\leq (d-1)/2, \;\gcd(m,d)=1\right\}.
    \]
\end{lem}
\begin{proof}
Given $d>0$, denote by $\cX_d$ the (smooth, complete) hyperelliptic curve corresponding to
$y^2=x^d+x^{-d}$. We have a commutative diagram
\begin{figure}[H]
\centering
\includegraphics[width=0.36\textwidth]{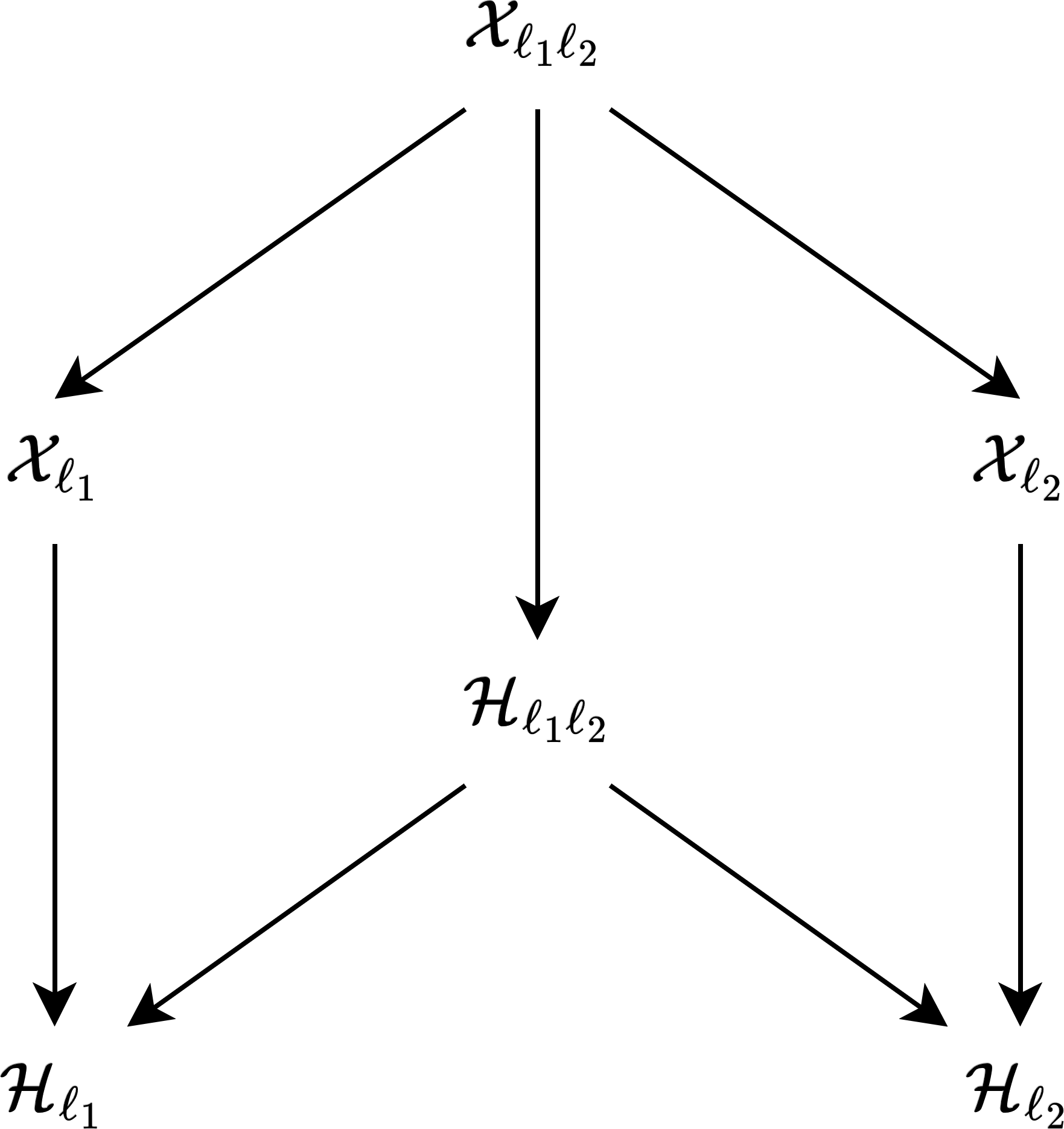}
\end{figure}
\noindent where the maps $\cX_{nm}\to \cX_m$ come from $(x,y)\mapsto (x^n,y)$ and
$\cX_n\to \cH_n$ from $(x,y)\mapsto (x+1/x,y)$ and
$\cH_{nm}\to\cH_m$ from $(x,y)\mapsto (\varphi_n(x),y)$.

We now reason in the spirit of \cite[\S3.1]{TTV}.
Define $\sigma,\zeta,\iota\in\text{Aut}(\cX_{\ell_1\ell_2})$ via
$\sigma(x,y)=(1/x,y)$ and $\zeta(x,y)=(\zeta_{\ell_1\ell_2}x,y)$ and
$\iota(x,y)=(-x,iy)$. Here $\zeta_n$ denotes a primitive $n$-th root
of unity and $i^2=-1$. Via pull-back, the spaces of regular $1$-forms on
$\cH_{\ell_j}$ and on $\cH_{\ell_1\ell_2}$ will be considered as
$\sigma$-invariant differentials on $\cX_{\ell_1\ell_2}$.
The latter space has basis
\[
\left\{ \omega_n:= (x^{n-1}-x^{-n-1})\frac{dx}{y}\;:\; 1\leq n \leq (\ell_1\ell_2-1)/2\right\}
\]
and it equals the space of
pull-backs of regular $1$-forms
on $\cH_{\ell_1\ell_2}$.
For $j=1,2$, the differentials related to $\cH_{\ell_j}$ are precisely those
spanned by
$\{\omega_n\colon \ell_{2-j}|n\}$.
Hence $\{\omega_n\colon \gcd(n,\ell_1\ell_2)=1\}$ spans a subspace
of (the pull-back of) $H^0(\cH_{\ell_1\ell_2},\Omega^1)$ complementary
to the differentials from both $\cH_{\ell_j}$.
Note that
\[ \iota^*\omega_n=(-1)^{n+1}i\cdot \omega_n\]
and
\[
\left(\zeta^*+(\zeta^{-1})^*\right)\omega_n=(\zeta_{\ell_1\ell_2}^n+\zeta_{\ell_1\ell_2}^{-n})\omega_n.
\]
Since $\text{span}\{\omega_n\colon \ell_1|n\}\cap \text{span}\{\omega_n\colon \ell_2|n\}=\{0\}$,
upto isogeny $\cJ(\cH_{\ell_1\ell_2})$
contains the product
$\cJ(\cH_{\ell_1})\times\cJ(\cH_{\ell_2})$. This results in the
asserted decomposition (defined over
$\mathbb{Q}$) of $\cJ(\cH_{\ell_1\ell_2})$,
with $\dim(A)=g(\cH_{\ell_1\ell_2})-g(\cH_{\ell_1})-g(\cH_{\ell_2})=(\ell_1-1)(\ell_2-1)/2$.

The action of
$\mathbb{Z}[\iota,\zeta+\zeta^{-1}]\subset \text{End}(\cJ(\cX_{\ell_1\ell_2}))$
fixes the parts coming from
$\cJ(\cH_{\ell_1}), \cJ(\cH_{\ell_2})$ and
$\cJ(\cH_{\ell_1\ell_2})$.
As a consequence, comparing
fixed subspaces, the abelian variety $A$ corresponds to
$\text{span}\{\omega_n\colon \gcd(n,\ell_1,\ell_2)=1\}.$
In terms of the standard
identification
\[ \text{Gal}(\mathbb{Q}(i,\zeta_{\ell_1\ell_2}+\zeta^{-1}_{\ell_1\ell_2})/\mathbb{Q})\cong
(\mathbb{Z}/4\ell_1\ell_2\mathbb{Z})^\times/\langle \overline{a}\rangle
\cong (\mathbb{Z}/4\mathbb{Z})^\times\times (\mathbb{Z}/\ell_1\ell_2\mathbb{Z})^\times/\langle\overline{\pm 1}\rangle
\]
(with $a\equiv 1\bmod 4$ and $a\equiv -1\bmod \ell_1\ell_2$),
the expressions for
$\iota^*\omega_n$ and 
$\left(\zeta^*+(\zeta^{-1})^*\right)\omega_n$ imply that
$A$ has CM with CM-field
$\mathbb{Q}(i,\zeta_{\ell_1\ell_2}+\zeta^{-1}_{\ell_1\ell_2})$,
and the corresponding CM type
is the asserted one.
\end{proof}

Lemma~\ref{ell1ell2} results in
an additional condition
for maximality of $\cH_{\ell_1\ell_2}$ over some
finite field of characteristic
$p\equiv 1\bmod 4$. The
conditions we have so far, are:
\begin{itemize}
    \item[(a)] $\cH_{\ell_1}$ and $\cH_{\ell_2}$ are maximal over some common
    $\FF_{p^n}$, so in particular $\ell_1,\ell_2\equiv 1\bmod 4$;
    \item[(b)] The CM abelian variety $A$ appearing in Lemma~\ref{ell1ell2} has
    supersingular reduction in characteristic $p$, in other words, the only slope of Frobenius at $p$ is $1/2$.
\end{itemize}
An immediate consequence of condition (a) is the following.
\begin{lem}\label{2val}
    If $\cH_{\ell_1}$ is maximal
    over $\FF_{p^{2m}}$ and $\cH_{\ell_2}$ is maximal
    over $\FF_{p^{2n}}$,
    then maximality of
    $\cH_{\ell_1\ell_2}$ in
    characteristic $p$ implies
    that $n$ and $m$ have equal $2$-adic valuation.
\end{lem}
\begin{proof}
    Denote the $2$-adic valuation of a nonzero integer $a$ by $v_2(a)$.
    If $\cH_{\ell_1\ell_2}/\FF_{p^{2k}}$ is maximal,
    then so is $\cH_{\ell_1}$
    and therefore $v_2(k)=v_2(m)$.
    By the same argument $v_2(k)=v_2(n)$.
\end{proof}

\begin{example}{\rm
Suppose that $\ell_1,\ell_2\equiv 1\bmod 4$ are primes, with
the property that
maximality of $\cH_{\ell_j}$
in characteristic $p\equiv 1\bmod 4$ only occurs when
$p$ is a primitive root modulo
$\ell_j, j=1,2$.
This condition is satisfied, for
example, if the $\ell_j$
are as described in Cor.~\ref{corol}. As remarked at the end of Section~\ref{primedeg}, it also
holds for $\ell_1,\ell_2\leq 101$.
Combining Lemma~\ref{2val}
and Thm.~\ref{1,1} yields that
under these conditions,
a necessary condition for
maximality of $\cH_{\ell_1\ell_2}$ in
characteristic $p\equiv 1\bmod 4$ (besides $p$ being
a primitive root modulo $\ell_1$ and $\ell_2$) is that
$v_2(\ell_1-1)=v_2(\ell_2-1)$.

For all pairs $\ell_1\neq \ell_2$ satisfying these
conditions and $\ell_1,\ell_2\leq 101$,
we checked using Magma that
the condition mentioned in (b)
above does not hold for
any $p\equiv 1\bmod 4$ which
is a primitive root modulo $\ell_1,\ell_2$.
In other words, we did not find
any $\cH_d$ with $d$ divisible
by at least two distinct prime numbers, such that $\cH_d$ attains
maximality in some characteristic $p\equiv 1\bmod 4$.
}\end{example}
\begin{question}\label{secondquestion}
    Is it true that a necessary condition for maximality of $\cH_d$ for $d>2$  in some characteristic $p\equiv 1\bmod 4$,
    is that $d$ is a power of a prime $\equiv 1\bmod 4$?
\end{question}

\vspace{\baselineskip}
Next, curves $\cH_{\ell^n}$ are discussed. An analogue of
Lemma~\ref{ell1ell2} is as follows.
\begin{lem}\label{elln}
    Let $\ell$ be an odd prime number and $n\in\mathbb{Z}_{\geq 2}$. The map $\cH_{\ell^n}\to\cH_{\ell^{n-1}}$
     coming from $\varphi_{\ell^n}(x)=\varphi_{\ell^{n-1}}(\varphi_{\ell}(x))$ leads to a $\mathbb{Q}$-isogeny
    \[ \cJ(\cH_{\ell^n})\sim \cJ(\cH_{\ell^{n-1}})\times B\]
    for some abelian variety $B/\mathbb{Q}$ of dimension $\ell^{n-1}(\ell-1)/2$.

    Write $d=\ell^n$. In characteristic $0$, this $B$ has complex multiplication with endomorphism algebra
    $\mathbb{Q}(i,\zeta_d+\zeta_d^{-1})$. Identifying $\mbox{\rm{Gal}}(\mathbb{Q}(i,\zeta_d+\zeta_d^{-1})/\mathbb{Q})$
    with $(\mathbb{Z}/4\mathbb{Z})^\times\times (\mathbb{Z}/d\mathbb{Z})^\times/(\pm 1)$,
    the corresponding CM type $\Phi$ is described by the following subset of this Galois group:
    \[
    \Phi=\left\{\left((-1)^{m+1},\pm m\right)\;:\; 1\leq m\leq (d-1)/2, \;\gcd(m,d)=1\right\}.
    \]
\end{lem}
\begin{proof}
With notations as in Lemma~\ref{ell1ell2} and its
proof, one has the diagram
\begin{figure}[H]
\centering
\includegraphics[width=0.3\textwidth]{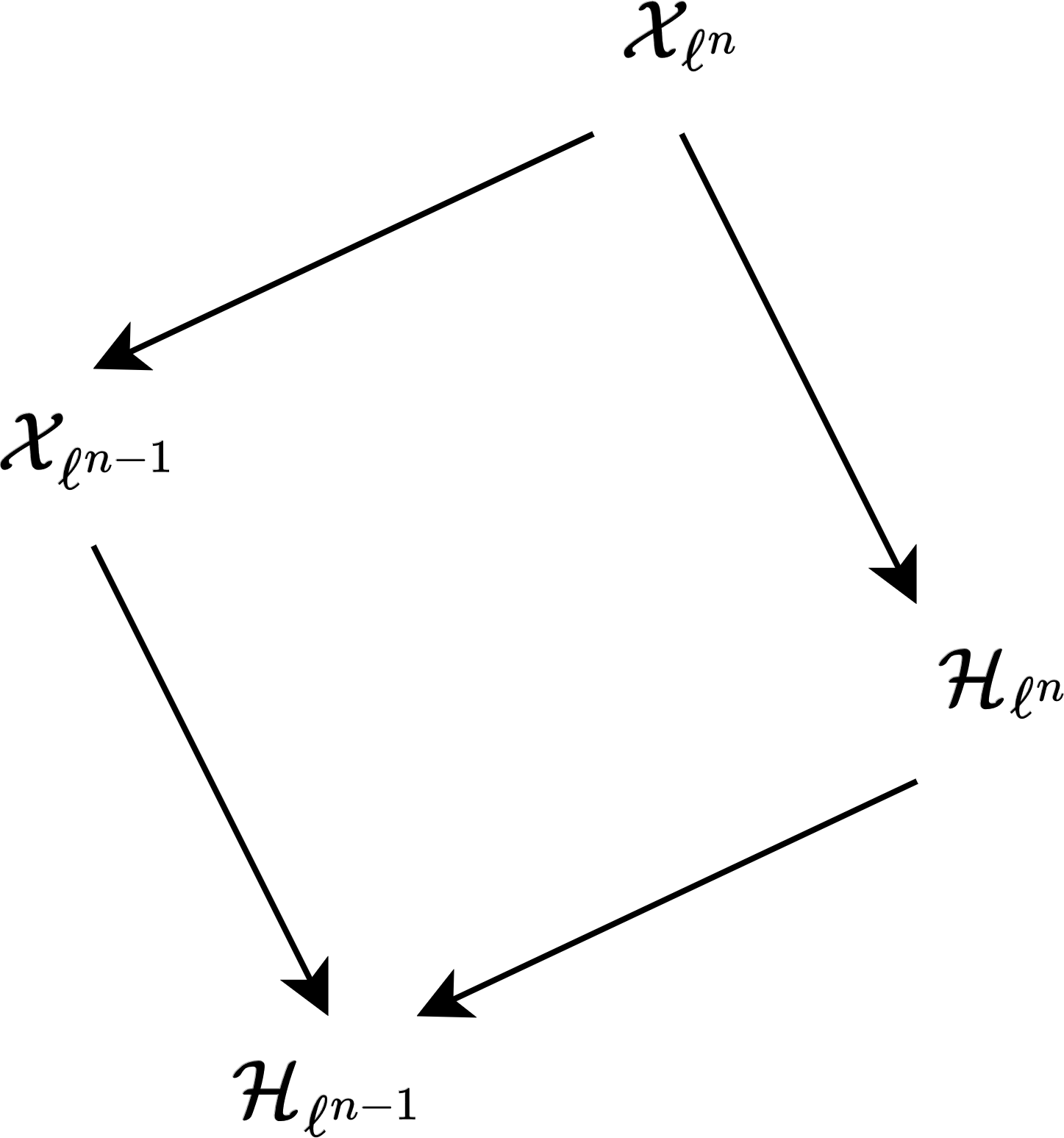}
\end{figure}
In $H^0(\cX_{\ell^n},\Omega^1)$
the $\sigma$-invariant differentials are those
coming from $\cH_{\ell^n}$
and they have as basis
\[
\left\{ \omega_n:= (x^{n-1}-x^{-n-1})\frac{dx}{y}\;:\; 1\leq n \leq (\ell^n-1)/2\right\}.
\]
The ones coming from $\cH_{\ell^{n-1}}$ are spanned
by $\{\omega_n\colon \ell^{n-1}|n\}$. The remainder
of the proof is similar to
that of Lemma~\ref{ell1ell2}.
\end{proof}

Note that the case $n=1$ (although excluded in the statement of the lemma) is
in fact the result of \cite[Prop.~4]{TTV}.
A difference between the current case and the one
described in Lemma~\ref{ell1ell2}, is that one can identify
cases such that $B$ has supersingular reduction:
\begin{pro}
With notations as in Lemma~\ref{elln},
if the prime $p\equiv 1\bmod 4$
generates $(\mathbb{Z}/\ell^n\mathbb{Z})^\times$
then $B$ has supersingular reduction at $p$.
\end{pro}
\begin{proof} Put $d=\ell^n$.
    By assumption,
    in
    $\textrm{Gal}(\mathbb{Q}(i,\zeta_d+\zeta_d^{-1})/\mathbb{Q})\cong (\mathbb{Z}/4\mathbb{Z})^\times\times (\mathbb{Z}/d\mathbb{Z})^\times/(\pm 1)$ the
    decomposition group $D_p$
    at $p$ equals
    $\{1\}\times (\mathbb{Z}/d\mathbb{Z})^\times/(\pm 1)$. As a consequence,
    $1/2$ is the only slope of
    Frobenius at $p$, proving the claim.
\end{proof}

\begin{remark}{\rm 
A test with Magma (considering all primes $\ell\equiv 1\bmod 4$ upto $101$ and $n=2$,
as well as $\ell\leq 29$ and $n=3$) showed that in these
cases, the only primes $p\equiv 1\bmod 4$ such that
$\cH_\ell$ attains maximality
in characteristic $p$ and moreover $B$ has supersingular
reduction at $p$, are the ones
for which $p$ generates
$(\mathbb{Z}/\ell^n\mathbb{Z})^\times$.

If this holds in general,
then it would provide a necessary condition for
maximality of $\cH_{\ell^n}$
in characteristic $p\equiv 1\bmod 4$.
}\end{remark}

\begin{example}{\rm
 For $\ell=5$ and $n=2$
 and every $p\in\{13, 17, 37, 53, 73, 97, 233, 277\}$
 (primes $1\bmod 4$ ranging over the generators modulo $25$), Magma verifies that
 the numerator of the zeta-function of $\cH_{25}/\FF_p$
 equals $(p^2x^4+1)(p^{10}x^{20}+1)$. Hence here $\cH_{25}/\FF_{p^{20}}$
 is maximal.

 Taking $\ell=5$ and $n=3$,
 the prime $p=13$ generates the
 units modulo $125$.
 According to Magma, the numerator of the
 zeta-function of $\cH_{125}/\FF_{13}$ is
 $(p^2x^4+1)(p^{10}x^{20}+1)(p^{50}x^{100}+1)$.
 As a consequence, $\cH_{125}/\FF_{13^{100}}$ is maximal!
Magma code used for these cases is presented on p.~\pageref{p=13}.
 Although additional examples
 of this kind seem beyond the
 power of Magma, the ones
 presented here motivates the
 following, generalizing Questions~\ref{prime1mod4case} and \ref{secondquestion}.}\end{example}

\begin{question}\label{mainquestion}
Let $\ell\equiv 1\bmod 4$ be prime and $n\in\mathbb{Z}_{\geq 1}$ and $p\equiv 1\bmod 4$ a prime
generating $(\mathbb{Z}/\ell^n\mathbb{Z})^\times$; write $m=\ell^{n-1}(\ell-1)$ for the order of this group. 
Is it true that the curve $\cH_{\ell^n}$ is maximal over $\FF_{p^m}$ (and hence over all odd degree extensions
of the latter field)?

Moreover, suppose $\cH_d$ is maximal over some field of characteristic $p\nmid 2d$ (and $d\geq 3$), with $p\equiv 1\bmod 4$.
Is it true that $d=\ell^n$ for some prime $\ell\equiv 1\bmod 4$ and $n\in\mathbb{Z}_{\geq 1}$?
\end{question}
The results and described experiments in this text are consistent with positive answers to both questions.
As an example, in the case $n=1$, Thm.~\ref{promax} provides
a positive answer to the first question posed here.

\section*{Appendix: some Magma code}
In this section some of the used Magma code (see
\cite{Magma} for the language) is presented.

\vspace{\baselineskip}
The following code 
computes the curve $\cH_d$
and checks maximality over $\FF_q$ (provided $d,q$ are reasonably small).\label{p=13}

\vspace{\baselineskip}
\begin{lstlisting}
PZ<x>:=PolynomialRing(Integers());
cheb := function(d)
   ch1:=x; ch2:=x^2-2;
   for m in [3..d] do   
     ch:=x*ch2-ch1; ch1:=ch2; ch2:=ch; 
   end for;
   return ch;
end function;

F:=GF(13^20); PF<x>:=PolynomialRing(F);
H:=HyperellipticCurve(PF!cheb(25));
SerreBound(H) eq #H;

\end{lstlisting}

\vspace{\baselineskip}
It turns out that some ``larger'' examples can be verified using the numerator
of the zeta function of $\cH_d/\FF_p$. Here is an example.

\vspace{\baselineskip}
\begin{lstlisting}
p:=13; F:=GF(p); PF<x>:=PolynomialRing(F); 
H:=HyperellipticCurve(PF!cheb(125));
Numerator(ZetaFunction(H));   
\end{lstlisting}

\vspace{\baselineskip}
Executing the lines above outputs $(p^2x^4+1)(p^{10}x^{20}+1)(p^{50}x^{100}+1)$
in case $d=125, p=13$.
As a consequence, $\cH_{125}$ is maximal over $\FF_{13^n}$ if and only if
$n=100m$ with $m$ odd.

\vspace{\baselineskip}
The next code computes the slopes $\frac{\#\left(\Phi_\ell\cap D_p\cdot \tau\right)}{\#D_p}$
used in Section~\ref{primedeg}.
\label{PrimeMagma}

\vspace{\baselineskip}
\begin{lstlisting}
mg:=function(m,n) return Min(m mod n,n-(m mod n));end function;
slopes:=function(ell, p)
    Phi:={<(-1)^(n+1), mg(n,ell)> : n in [1..(ell-1)div 2]};
    ord:=Modorder(p,ell); sl:={};
    Dp:={<(-1)^(((p-1)div 2)*n), mg(p^n,ell)> : n in [1..ord]};
    for m in [1..(ell-1)div 2]
     do Dpmp:= { <x[1], mg(m*x[2],ell)> : x in Dp};
        Include(~sl, #{x: x in Dpmp | x in Phi}/#Dp);
        Dpmm:={ <-x[1], mg(m*x[2],ell)> : x in Dp};
        Include(~sl, #{x: x in Dpmm | x in Phi}/#Dp);
    end for;
    return sl;
end function;
\end{lstlisting}

  The next code verifies that
  the only characteristics $p\equiv 1\bmod 4$ such that
  $\cH_{89}$ attains maximality
  over extensions of $\FF_p$, are the ones for which
  $p$ is a primitive root modulo $89$:

\begin{lstlisting}
ell:=89;  
classes:={p : p in [1..4*ell-1] | Gcd(p,4*ell) eq 1};                        
for p in classes do
  if slopes(ell,p) eq {1/2} then
     if p mod 4 eq 1 then <p , Modorder(p,ell)>; end if; 
  end if;
end for;  
\end{lstlisting}

Below is code computing the slopes of Frobenius at $p$ for
the CM abelian varieties described in Lemma's \ref{ell1ell2}
and \ref{elln}.

\vspace{\baselineskip}
\begin{lstlisting}
slopes2 := function(d, p)
    enns:= {n : n in [1..(d-1)div 2] | Gcd(n,d) eq 1};
    Phi:= {<(-1)^(n+1), mg(n,d)> : n in enns};
    ord:= Modorder(p,d); sl:={};
    Dp:= {<(-1)^(((p-1)div 2)*n), mg(p^n,d)> : n in [1..ord]};
    for m in [1..(d-1)div 2] do 
      if Gcd(m,d) eq 1 then
        Dpmp:= { <x[1], mg(m*x[2],d)> : x in Dp};
        Include(~sl, #{x: x in Dpmp | x in Phi}/#Dp);
        Dpmm:= { <-x[1], mg(m*x[2],d)> : x in Dp};
        Include(~sl, #{x: x in Dpmm | x in Phi}/#Dp);
      end if;
    end for;
    return sl;
end function;
\end{lstlisting}

Finally, we present code used to check whether a given
$\cH_{\ell_1\ell_2}$ might attain maximality in some
characteristic $p\equiv 1\bmod 4$. Following the discussion
in \S\ref{1,1} it is assumed (verified) that a
necessary condition is that $p$ is a primitive root
modulo both $\ell_j$'s. For primes $p$ in the congruence classes
satisfying this, the code checks whether the abelian
variety $A$ described in Lemma~\ref{ell1ell2} has
supersingular reduction at $p$ (so, set of slopes of
Frobenius equal to $\{1/2\}$).
\begin{lstlisting}
primroots := function(p)
    a:=PrimitiveRoot(p);
    return {a^m mod p : m in [1..p-2] | Gcd(m,p-1) eq 1};
end function;
check := function(ell1, ell2)
  s1:=primroots(ell1); s2:=primroots(ell2); ss:={};
  cls:={ CRT([a,b,1], [ell1,ell2,4]) : a in s1, b in s2 };
  for p in cls do
    if slopes2(ell1*ell2,p)eq{1/2} then Include(~ss,p); end if;
  end for;
  return ss;
end function;
\end{lstlisting}
\normalsize{\textbf{Acknowledgment.} 
 \rm{We thank J.D.~Top for making the two diagrams used
 in \S\ref{subsec1mod4}. The first author was partially supported by FAPESP grant No. 2023/08271-5.}

\end{document}